\documentclass[reqno,12pt]{amsart}

\usepackage[margin=1.2in]{geometry}

	\usepackage{comment}

	\usepackage[utf8]{inputenc}

	\usepackage{amsmath,amsfonts,amssymb,amsthm}

	\usepackage[T1]{fontenc}

	\usepackage{etoolbox}
	\patchcmd{\section}{\scshape}{\scshape\bfseries}{}{}
	\makeatletter
	\renewcommand{\@secnumfont}{\scshape\bfseries}
	\makeatother

	\let\emptyset\varnothing
	
	\let\epsilon\varepsilon

	\usepackage{parskip}

	\usepackage[bookmarks=true,bookmarksopen=true]{hyperref}

	\newtheorem{theorem}{Theorem}

	\newtheorem{corollary}[theorem]{Corollary}
        
	\newtheorem{lemma}[theorem]{Lemma}
	
	\newtheorem{main}{Theorem}

	\numberwithin{theorem}{section}
	
	\theoremstyle{plain}

	\theoremstyle{definition}

	\theoremstyle{remark}

	\numberwithin{equation}{section}
	
	\numberwithin{table}{section}

	\usepackage{tikz}
	\usetikzlibrary{matrix}

    \usepackage[all,cmtip]{xy}

	\usepackage{todonotes}

	\usepackage{enumerate}

	\usepackage{float}

	\makeatletter
	\def\blfootnote{\gdef\@thefnmark{}\@footnotetext}
	\makeatother

        \DeclareMathOperator{\sys}{sys}

	\newcommand{\bR}{\mathbb{R}}
	
	\newcommand{\bZ}{\mathbb{Z}}	\newcommand{\Z}{\bZ}

	\mathcode`l="8000
	\begingroup
	\makeatletter
	\lccode`\~=`\l
	\DeclareMathSymbol{\lsb@l}{\mathalpha}{letters}{`l}
	\lowercase{\gdef~{\ifnum\the\mathgroup=\m@ne \ell \else \lsb@l \fi}}%
	\endgroup

\usepackage{nicematrix}

\input{bordersquare.tex}
\date{\today}
\usepackage{appendix}
\usepackage{multicol}
\usepackage{subcaption}
\usetikzlibrary{calc}

\title{Higher cosystoles of matroids}

\author{James Dylan Douthitt}
\address{Department of Mathematics, Syracuse University}
\email{jddouthi@syr.edu}

\author{Elana Israel}
\address{Department of Mathematics, Statistics, and Computer Science, St. Olaf College}
\email{israel3@stolaf.edu}

\author{Lee Kennard}
\address{Department of Mathematics, Syracuse University}
\email{ltkennar@syr.edu}

\begin{document}

\begin{abstract}
We define a matroid invariant called the three-cosystole that is related to higher notions of cogirth for weighted matroids, and we prove an optimal upper bound for it in the class of regular matroids of rank at most six. To accomplish this, we show that it is increasing under matroid extensions and then estimate it for each of the maximal simple regular matroids of rank at most six.
\end{abstract}

\maketitle

A weighted matroid is a matroid $M$ equipped with a non-negative and non-zero function $\mu:E \to \bR_{\geq 0}$ on its ground set $E$. We define a normalized, weighted girth by the formula
	\[\sys(M,\mu) = \min \mu(C)/\mu(E),\]
where the minimum runs over circuits $C$, and where $\mu(X)$ for a subset $X \subseteq E$ denotes the sum of $\mu(e)$ over $e \in X$. The {\it systole} of $M$ is the matroid invariant 
    \[\sys(M) = \max \sys(M, \mu)\]
obtained by maximizing over admissible weight functions. The {\it cosystole} $\sys^*(M)$ is defined similarly using cocircuits in place of circuits. Matroid cosystoles have been studied in \cite{ChoChenDing07} in general, \cite{CrenshawOxley23} in the binary case, and \cite{KWW2} in the regular case. An advantage of (co)systoles over (co)girths is their invariance under scalings of the weight function. In particular, we may restrict to the case where $\mu(E) = 1$.

The main results in \cite{KWW2} are in Riemannian geometry, but the main computation to prove them is to show sharp upper bounds of the form $\sys^*(M) \leq s^*(d)$ for regular matroids of rank $d \leq 9$. 
For rank six, the estimate states that $\sys^*(M) \leq 1/3$. Moreover, Theorem 3.14 in that paper shows that, for any $\mu$ with total weight one, there exist six pairwise distinct cocircuits $C_i^*$ whose $\mu$-weights sum to at most two. In addition to implying $\sys^*(M) \leq 1/3$, it follows using that $M$ is binary that there exist three linearly independent cocircuits whose $\mu$-weights sum to at most one.

In this paper, we reprove the second of these estimates using a different strategy. This has three advantages. The first is the following, formal strengthening of the estimate on linearly independent cocircuits.

\begin{main}\label{thm:3-cogirth}
If $M$ is a regular matroid of rank six and $\mu$ is a probability distribution on $E(M)$, 
then there exist cocircuits $C_1^*$, $C_2^*$, and $C_3^*$ satisfying the estimate
	\[\mu(C_1^*) + \mu(C_2^*) + \mu(C_3^*) \leq 1\]
and with the property that $C_h^* \not\subseteq C_i^* \cup C_j^*$ for all $\{h, i, j\} = \{1, 2, 3\}$.
\end{main}

Theorem \ref{thm:3-cogirth} can be rephrased as an upper bound on the matroid invariant
    \[\sys_3^*(M) = \max_\mu \, \min \, \sum_{i=1}^{3} \mu(C_i^*)\]
that we call the $3$-cosystole, where the maximum runs over weight functions $\mu:E(M) \to \mathbb{R}_{\geq 0}$ on the ground set with $\mu(E) = 1$, and where the minimum runs over all triples of cocircuits $\{C_1^*, C_2^*, C_3^*\}$ with the property that no $C_i^*$ is contained in the union of the other two. This {\it non-inclusion property} is stronger than requiring only that the cocircuits $C_i^*$ are pairwise distinct (equivalently that $C_i^* \not\subseteq C_j^*$ for all $i \neq j$).

The second outcome of the proof is fundamental to the proof strategy, which involves a reduction to the case of {\it maximal simple regular} matroids. This is achieved by showing that higher cosystoles are increasing under equal rank extensions. This result can be rephrased in terms of deletions and then extended to arbitrary minors.

\begin{main}[Monotonicity]\label{thm:monotonicity}
Suppose $M$ is a regular matroid. If $e$ is an element of the ground set that is not a coloop and $r(M/e)\geq 3$, then 
    \[{\sys}_3^*(M\backslash e)
    \leq {\sys}_3^*(M)
    \leq {\sys}_3^*(M/e).\]
Similarly, if $\operatorname{si}(M)$ denotes the simplification of $M$, then
    \[\sys_3^*(M) = \sys_3^*(\operatorname{si}(M)).\]
\end{main}

The third outcome of the proof is a listing of the maximal simple regular matroids in rank six and a estimate of the invariant $\sys_3^*(M)$ in each case. It turns out there are exactly $11$ such matroids, as can be seen by using Seymour's classification (see \cite{Seymour80, DanilovGrishukhin99}). We describe these matroids in the notation of \cite{Oxley-book}, highlight properties relevant to our computations, and outline the second author's recent, independent proof of this list in \cite{Israel-PhD}, which highlights the curious property that exactly one of the {\it maximal simple cographic} matroids in rank six is not {\it maximal simple regular} (see also \cite[Page 424]{DanilovGrishukhin99}). Rank six is the first time where this happens. Rank six is also the first time we see a decomposable regular matroid that is not also graphic, cographic, or sporadic, so the list an interesting example to be paired with Seymour's theorem.

Given the list, the remainder of the work in this paper is to prove the desired estimates on the $11$ maximal simple regular matroids of rank six. The result is as follows. 

\begin{main}\label{thm:estimates}
If $M$ is a maximal simple regular matroid of rank six, then one of the following holds:
    \begin{enumerate}
        \item $M$ is graphic and $\sys_3^*(M) = \tfrac 6 7$.
        \item $M$ is cographic and $\sys_3^*(M) \leq 1$. 
        \item $M$ is the parallel connection $P(M(K_3), R_{10})$ or the generalized parallel connection $P_T(M(K_5), M^*(K_{3,3}))$ and, in either case, $\sys_3^*(M) = \tfrac{12}{13}$.
    \end{enumerate}
\end{main}

In particular, Theorem \ref{thm:estimates} implies $\sys_3^*(M) \leq 1$ for maximal simple regular matroids of rank six. Theorem \ref{thm:monotonicity} extends this estimate to all regular matroids of rank six and hence implies Theorem \ref{thm:3-cogirth}. 

We also note that putting constant weights on the edges gives a lower bound in terms of unweighted matroid invariants:
    \[\sys_3^*(M) \geq \sys_3^*(M, \mu_1) = \min \frac{|C_1^*| + |C_2^*| + |C_3^*|}{|E(M)|},\]
where $|C_i^*|$ denotes the cardinality of the cocircuit $C_i^*$, and the minimum again runs over the triples of cocircuits satisfying the non-inclusion property. For the cographic matroid on the Petersen graph, $G_1$ in Figure~\ref{fig:Graphs}, all cocircuits have at least five elements and the ground set has $15$, so this implies
    \[\sys_3^*(M^*(G_1)) \geq 1.\]
Therefore the upper bound of one is optimal for regular matroids of rank six. 

In general, one can combine a simple averaging argument together with an estimate as in the previous paragraph to get the estimates
    \[\sys_3^*(M) \geq 3 \sys^*(M) \quad\text{and}\quad 
    \sys^*(M) \geq \frac{g^*(M)}{|E(M)|},\]
where $g^*(M)$ denotes the cogirth of $M$. While equality holds in both cases for the cographic matroid on the Peterson graph, the fact that $\sys^*(M)$ is a maximum of relative weighted cogirths allows for $\sys^*(M) > g^*(M) / |E(M)|$. The first time we see this for a maximal simple regular matroid is for a pair of cographic matroids in rank five. The next theorem includes this calculation and, for completeness, the $3$-cosystole calculation for every maximal simple regular matroid of rank less than six. 

\begin{main}\label{thm:estimates<6}
If $M$ is a maximal simple regular matroid of rank $d \in \{3, 4, 5\}$, then one of the following holds:
    \begin{enumerate}
        \item $M$ is the graphic matroid $M(K_{d+1})$ and $\sys_3^*(M) = \tfrac{6}{d+1}$.
        \item $M$ is the cographic matroid $M^*(K_{3,3})$ of rank four and $\sys_3^*(M) = \tfrac 4 3$.
        \item $M$ is the cographic matroid $M^*(G_{5,3})$ of rank five and $\sys_3^*(M) = \tfrac{12}{11}$.
        \item $M$ is the cographic matroid $M^*(G_{5,4})$ of rank five and $\sys_3^*(M) = \tfrac 9 8$.
        \item $M$ is the sporadic matroid $R_{10}$ of rank five and $\sys_3^*(M) = \tfrac 6 5$.
    \end{enumerate}
Graphs $G_{5,3}$ and $G_{5,4}$ are those given in Figure~\ref{fig:2Graphs}. 
\end{main}

\begin{figure}
\begin{subfigure}{0.32\textwidth}
\centering
\begin{tikzpicture}[scale=.75]
\def\rhex{1}
\def\rout{2}
\def\n{6}
\draw[line width=1pt] (0,0) circle (\rout);
\foreach \i in {1,...,\n}{
    \node[circle,fill=black,inner sep=1.5pt]
        (H\i) at ({60*(\i-1)}:\rhex) {};
}
\foreach \i [evaluate=\i as \j using {int(mod(\i,\n)+1)}] in {1,...,\n}{
    \draw[line width=1pt] (H\i) -- (H\j);
}
\foreach \i in {1,...,\n}{
    \draw[line width=1pt]
        (H\i) -- ({60*(\i-1)}:\rout);
}
\node[circle,fill=black,inner sep=1.5pt]
    (TR) at ($(H1)!0.5!(H2)$) {};

\node[circle,fill=black,inner sep=1.5pt]
    (BR) at ($(H6)!0.5!(H1)$) {};
\draw[line width=1pt] (BR)to[in=-135,out=135](TR);
\end{tikzpicture}
\subcaption*{$G_{5,3}$}
\end{subfigure}
    \begin{subfigure}{0.32\textwidth}

\centering
\begin{tikzpicture}[scale=.75]

\def\rhex{1}
\def\rhexy{.75}
\def\rout{2}
\def\n{8}
\draw[line width=1pt] (0,0) circle (\rout);
\foreach \i in {1,...,\n}{
    \node[circle,fill=black,inner sep=1.5pt]
    (H\i) at ({360/8*(\i-1)}:\rhex){} ;
}

\foreach \i in {1,...,\n}{
    \node[circle, fill=none,inner sep=1.5pt]
    (G\i) at ({360/8*(\i-1)}:\rhexy){\tiny $\i$} ;
}

\foreach \i [evaluate=\i as \j using {int(mod(\i,\n)+1)}] in {1,...,\n}{
    \draw[line width=1pt] (H\i) -- (H\j);
}
\foreach \i in {1,...,\n}{
    \draw[line width=1pt]
        (H\i) -- ({360/8*(\i-1)}:\rout);
}
\end{tikzpicture}
\subcaption*{$G_{5,4}$}
\end{subfigure}
\caption{Graphs $G_{5,3}$ and $G_{5,4}$ drawn on $\mathbb{P}_2$}\label{fig:2Graphs}
\end{figure}
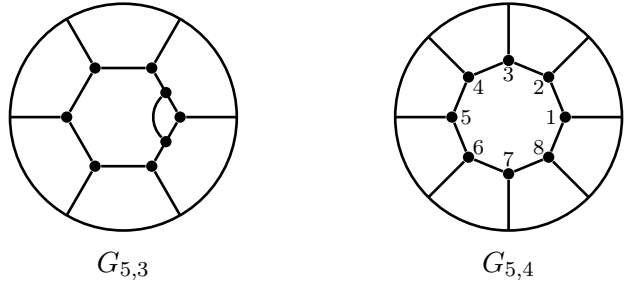

Notice that the estimate $\sys_3^*(M) \leq 1$ fails for ranks three, four, and five. Note also that $\sys_3^*(M)$ is not defined in general for ranks one or two since, for example, there is no triple of cocircuits in $M^*(K_3)$ satisfying the non-inclusion property.  This underscores the observations made in \cite{KWW2} that the proof method for the geometric results does not work if one lowers the torus rank below six.

\subsection*{Organization}
In Section \ref{sec:list}, we list all regular matroids of rank at most six, give embeddings of the relevant graphs in the maximal simple cographic cases, and prove exactly one of the maximal simple cographic matroids is not maximal simple regular. 
In Section \ref{sec:HigherCogirths}, we give precise definitions of the three-cosystole and prove Theorem \ref{thm:monotonicity}. 
In Section \ref{sec:estimates}, we prove the desired estimates in the maximal simple regular case, thereby proving Theorems \ref{thm:estimates} and \ref{thm:estimates<6}. As stated earlier, Theorem \ref{thm:3-cogirth} follows immediately.

\subsection*{Acknowledgements}
Part of this work was completed as part of the second author's PhD thesis, and she is grateful for support from NSF Grant DMS-2005280. The third author is partially supported by NSF Grant DMS-2402129.

\vspace{1em}
\section{Regular matroids of rank at most six}
\label{sec:list}
\vspace{1em}

Regular matroids play an important role in matroid theory and have been characterized in multiple ways. A deep and constructive characterization of regular matroids is due to Seymour \cite{Seymour80}: {\it Regular matroids are built from $1$-sums, $2$-sums, and $3$-sums of graphic matroids, cographic matroids, and the sporadic matroid $R_{10}$}. In small ranks, it is possible to use Seymour's classification to construct an explicit list of regular matroids. We do this in ranks up to six in this section.

We first note that adding loops or elements in parallel to existing elements of the ground set makes the ground set larger while preserving the rank and regularity of the matroid. As these extensions are trivial and lead to an infinite set of objects, we restrict to {\it simple} matroids to avoid loops or parallel elements.

In rank one, the uniform matroid $U_{1,1} = M(K_2)$ is the only simple matroid, and therefore the only simple regular matroid.

In rank two, there are two simple regular matroids: the direct sum $U_{2,2}=M(K_2) \oplus M(K_2)$ with two elements and the graphic matroid $U_{2,3}=M(K_3)$ with three. Note that the first of these is isomorphic to $M(P_2)$ where $P_2$ is the two-edge path and is therefore isomorphic to a restriction of $M(K_3)$.

In rank three, there are five simple regular matroids. All of them are graphic and representable as the graphic matroid $M(G)$ for some subgraph $G \subseteq K_4$. 

In rank four, there are 17 simple regular matroids. All but one is a restriction of $M(K_5)$, and the missing one is the cographic (and non-graphic) matroid $M^*(K_{3,3})$ (see \cite{ErdahlRyshkov94,Nienhaus-pre}).

In rank five, there are many simple regular matroids. As in smaller ranks, most are restrictions of the maximal graphic matroid $M(K_6)$. There are additionally three non-graphic simple matroids that are maximal in the sense of not being restrictions of simple regular matroids of rank six. They are the cographic matroids $M^*(G_{5,3})$ and $M^*(G_{5,4})$ and the sporadic matroid $R_{10}$, where $G_{5,3}$ and $G_{5,4}$ are the unique cubic, non-planar graphs on eight vertices with girth $3$ and $4$, respectively (see Figure~\ref{fig:2Graphs}), and where $R_{10}$ is the unique $10$-element matroid that is regular but neither graphic nor cographic (see \cite{ErdahlRyshkov94} and \cite[Example 1.7]{KWW2}). 

In view of these lists, it is desirable to restrict our attention to {\it maximal simple regular matroids}. The number of these in ranks one through five is $1$, $1$, $1$, $2$, and $4$. In rank six, \cite{DanilovGrishukhin99} (see also \cite{Israel-PhD}) proved the following.

\begin{theorem}[Maximal simple regular matroids in rank six]\label{thm:MSR}
Every simple regular matroid of rank six is contained in one of the following:
	\begin{itemize}
	\item the graphic matroid $M(K_7)$, 
	\item the cographic matroid $M^*(G_i)$ for some $i \in \{1, \ldots, 9\} \setminus \{7\}$, where the graphs $G_i$ are shown in Figure \ref{fig:Graphs},
	\item the $12$-element parallel connection $P(M(K_3), R_{10})$ of $M(K_3)$ and $R_{10}$, which is represented by the matrix in \eqref{equ:A12}, or
	\item the $16$-element generalized parallel connection $P_T(M(K_5), M^*(K_{3,3}))$ of $M(K_5)$ and $M(K_{3,3})$ across a triangle $T$, which is represented by the matrix in \eqref{equ:A16}.
	\end{itemize}
\end{theorem}

The last two matroids are a $2$-sum and a $3$-sum in the sense of Seymour, except that we first add an edge or three in parallel before performing the sum. In the notation of \cite[Sections 7.1 and 11.4]{Oxley-book}, these are called the parallel connection and generalized parallel connection over a triangle, respectively. Unless otherwise stated, $T$ will represent a triangle, that is, $T\cong U_{2,3}$ in the generalized parallel connection. A representation of $P(M(K_3),R_{10})$ over $\mathbb{Z}_2$ is given by the matrix $A_{12}$ in (\ref{equ:A12}). We note that in (\ref{equ:A12}) the submatrix corresponding to rows $v_1$ to $v_5$ and columns $f_2$ to $f_{11}$ is a matrix representation for $R_{10}$, which will be used in Section~\ref{sec:estimates} to prove Theorem~\ref{thm:estimates<6}. 
    \begin{equation}\label{equ:A12}
    A_{12} = 
    \bordersquare{
        &f_0 &f_1 & f_2 & f_3  & f_4 & f_5 & f_6 & f_7 &f_8 &f_9 &f_{10} &f_{11} \cr
        v_0&1 &1 & 0 & 0  & 0 & 0 & 0 & 0 &0 &0 &0  &0 \cr 
        v_1&1 &0 & 1 & 0  & 0 & 0 & 0 & 1 &1 &0 &0  &1 \cr
        v_2&0 &0 & 0 & 1  & 0 & 0 & 0 & 1 &1 &1 &0  &0 \cr
        v_3&0 &0 & 0 & 0  & 1 & 0 & 0 & 0 &1 &1 &1  &0 \cr
        v_4&0 &0 & 0 & 0  & 0 & 1 & 0 & 0 &0 &1 &1  &1 \cr
        v_5&0 &0 & 0 & 0  & 0 & 0 & 1 & 1 &0 &0 &1  &1 \cr
    }
    \end{equation}
The graphs $G_1,\ldots,G_9$, shown in Figure~\ref{fig:Graphs}, are precisely the set of $3$-edge-connected, cubic, non-planar graphs on 15 edges. An important property of eight of these graphs that we use in our computations is that, except for $G_9$, they admit embeddings into the real projective plane $\mathbb P_2$. While this follows by inspection of \cite{GloverHuneke75}, we give explicit embeddings here for simplicity.

\begin{figure}
\centering
\begin{subfigure}{0.32\textwidth}
\centering
\begin{tikzpicture}[scale=.75]

\def\rhex{1}
\def\rout{2}
\def\n{6}
\draw[line width=1pt] (0,0) circle (\rout);
\foreach \i in {1,...,\n}{
    \node[circle,fill=black,inner sep=1.5pt]
        (H\i) at ({60*(\i-1)}:\rhex) {};
}
\foreach \i [evaluate=\i as \j using {int(mod(\i,\n)+1)}] in {1,...,\n}{
    \draw[line width=1pt] (H\i) -- (H\j);
}
\foreach \i in {1,...,\n}{
    \draw[line width=1pt]
        (H\i) -- ({60*(\i-1)}:\rout);
}
\node[circle,fill=black,inner sep=1.5pt]
    (TM) at ($(H2)!0.5!(H3)$) {};
\node[circle,fill=black,inner sep=1.5pt]
    (BL) at ($(H4)!0.5!(H5)$) {};
\node[circle,fill=black,inner sep=1.5pt]
    (BR) at ($(H6)!0.5!(H1)$) {};
\node[circle,fill=black,inner sep=1.5pt]
    (M) at ($(H4)!0.5!(H1)$) {};
\draw[line width=1pt] (BL)--(M);
\draw[line width=1pt] (BR)--(M);
\draw[line width=1pt] (TM)--(M);

\end{tikzpicture}
\subcaption*{$G_1$}
\end{subfigure}
\begin{subfigure}{0.32\textwidth}
\centering
\begin{tikzpicture}[scale=.75]

\def\rhex{1}
\def\rout{2}
\def\n{6}
\draw[line width=1pt] (0,0) circle (\rout);
\foreach \i in {1,...,\n}{
    \node[circle,fill=black,inner sep=1.5pt]
        (H\i) at ({60*(\i-1)}:\rhex) {};
}
\foreach \i [evaluate=\i as \j using {int(mod(\i,\n)+1)}] in {1,...,\n}{
    \draw[line width=1pt] (H\i) -- (H\j);
}
\foreach \i in {1,...,\n}{
    \draw[line width=1pt]
        (H\i) -- ({60*(\i-1)}:\rout);
}

\node[circle,fill=black,inner sep=1.5pt]
    (TR) at ($(H1)!0.5!(H2)$) {};
\node[circle,fill=black,inner sep=1.5pt]
    (TL) at ($(H3)!0.5!(H4)$) {};
\node[circle,fill=black,inner sep=1.5pt]
    (BL) at ($(H4)!0.5!(H5)$) {};
\node[circle,fill=black,inner sep=1.5pt]
    (BR) at ($(H6)!0.5!(H1)$) {};

\draw[line width=1pt]
        (TR) -- ({30}:\rout);
\draw[line width=1pt]
        (TL) -- ({150}:\rout);
\draw[line width=1pt]
        (BL) -- ({210}:\rout);
\draw[line width=1pt]
        (BR) -- ({330}:\rout);

\end{tikzpicture}
\subcaption*{$G_2$}\label{fig:G2}
\end{subfigure}
\begin{subfigure}{0.32\textwidth}
\centering
\begin{tikzpicture}[scale=.75]

\def\rhex{1}
\def\rout{2}
\def\n{6}
\draw[line width=1pt] (0,0) circle (\rout);
\foreach \i in {1,...,\n}{
    \node[circle,fill=black,inner sep=1.5pt]
        (H\i) at ({60*(\i-1)}:\rhex) {};
}
\foreach \i [evaluate=\i as \j using {int(mod(\i,\n)+1)}] in {1,...,\n}{
    \draw[line width=1pt] (H\i) -- (H\j);
}
\foreach \i in {1,...,\n}{
    \draw[line width=1pt]
        (H\i) -- ({60*(\i-1)}:\rout);
}
\node[circle,fill=black,inner sep=1.5pt]
    (TR) at ($(H1)!0.5!(H2)$) {};
\node[circle,fill=black,inner sep=1.5pt]
    (TL) at ($(H3)!0.5!(H4)$) {};
\node[circle,fill=black,inner sep=1.5pt]
    (BL) at ($(H5)!0.33!(H6)$) {};
\node[circle,fill=black,inner sep=1.5pt]
    (BR) at ($(H5)!0.66!(H6)$) {};
\draw[line width=1pt] (BL)to[out=80, in=-30](TL);
\draw[line width=1pt] (BR)to[in=210, out=100](TR);
\end{tikzpicture}
\subcaption*{$G_3$}
\end{subfigure}
\begin{subfigure}{0.32\textwidth}
\centering
\begin{tikzpicture}[scale=.75]
\def\rhex{1}
\def\rout{2}
\def\n{6}
\draw[line width=1pt] (0,0) circle (\rout);
\foreach \i in {1,...,\n}{
    \node[circle,fill=black,inner sep=1.5pt]
        (H\i) at ({60*(\i-1)}:\rhex) {};
}
\foreach \i [evaluate=\i as \j using {int(mod(\i,\n)+1)}] in {1,...,\n}{
    \draw[line width=1pt] (H\i) -- (H\j);
}
\foreach \i in {1,...,\n}{
    \draw[line width=1pt]
        (H\i) -- ({60*(\i-1)}:\rout);
}
\node[circle,fill=black,inner sep=1.5pt]
    (TR) at ($(H2)!0.33!(H3)$) {};
\node[circle,fill=black,inner sep=1.5pt]
    (TL) at ($(H2)!0.66!(H3)$) {};
\node[circle,fill=black,inner sep=1.5pt]
    (BL) at ($(H5)!0.33!(H6)$) {};
\node[circle,fill=black,inner sep=1.5pt]
    (BR) at ($(H5)!0.66!(H6)$) {};
\draw[line width=1pt] (BL)--(TL);
\draw[line width=1pt] (BR)--(TR);
\end{tikzpicture}
\subcaption*{$G_4$}
\end{subfigure}
\begin{subfigure}{0.32\textwidth}
\centering
\begin{tikzpicture}[scale=.75]

\def\rhex{1}
\def\rout{2}
\def\n{6}
\draw[line width=1pt] (0,0) circle (\rout);
\foreach \i in {1,...,\n}{
    \node[circle,fill=black,inner sep=1.5pt]
        (H\i) at ({60*(\i-1)}:\rhex) {};
}
\foreach \i [evaluate=\i as \j using {int(mod(\i,\n)+1)}] in {1,...,\n}{
    \draw[line width=1pt] (H\i) -- (H\j);
}
\foreach \i in {1,...,\n}{
    \draw[line width=1pt]
        (H\i) -- ({60*(\i-1)}:\rout);
}
\node[circle,fill=black,inner sep=1.5pt]
    (TR) at ($(H1)!0.5!(H6)$) {};
\node[circle,fill=black,inner sep=1.5pt]
    (TL) at ($(H3)!0.5!(H4)$) {};
\node[circle,fill=black,inner sep=1.5pt]
    (BL) at ($(H4)!0.5!(H5)$) {};
\node[circle,fill=black,inner sep=1.5pt]
    (BR) at ($(H5)!0.5!(H6)$) {};
\draw[line width=1pt] (BL)to[out=45,in=-45](TL);
\draw[line width=1pt] (BR)to[in=-195,out=75](TR);
\end{tikzpicture}
\subcaption*{$G_5$}
\end{subfigure}
\begin{subfigure}{0.32\textwidth}
\centering
\begin{tikzpicture}[scale=.75]
\def\rhex{1}
\def\rout{2}
\def\n{6}
\draw[line width=1pt] (0,0) circle (\rout);
\foreach \i in {1,...,\n}{
    \node[circle,fill=black,inner sep=1.5pt]
        (H\i) at ({60*(\i-1)}:\rhex) {};
}
\foreach \i [evaluate=\i as \j using {int(mod(\i,\n)+1)}] in {1,...,\n}{
    \draw[line width=1pt] (H\i) -- (H\j);
}
\foreach \i in {1,...,\n}{
    \draw[line width=1pt]
        (H\i) -- ({60*(\i-1)}:\rout);
}
\node[circle,fill=black,inner sep=1.5pt]
    (TR) at ($(H1)!0.5!(H2)$) {};
\node[circle,fill=black,inner sep=1.5pt]
    (TL) at ($(H3)!0.5!(H4)$) {};
\node[circle,fill=black,inner sep=1.5pt]
    (BL) at ($(H4)!0.5!(H5)$) {};
\node[circle,fill=black,inner sep=1.5pt]
    (BR) at ($(H6)!0.5!(H1)$) {};
\draw[line width=1pt] (BL)to[in=-45,out=45](TL);
\draw[line width=1pt] (BR)to[in=-135,out=135](TR);
\end{tikzpicture}
\subcaption*{$G_6$}
\end{subfigure}
\begin{subfigure}{0.32\textwidth}
\centering
\begin{tikzpicture}[scale=.75]
\def\rhex{1}
\def\rout{2}
\def\n{6}
\draw[line width=1pt] (0,0) circle (\rout);

\foreach \i in {1,...,\n}{
    \node[circle,fill=black,inner sep=1.5pt]
        (H\i) at ({60*(\i-1)}:\rhex) {};
}
\foreach \i [evaluate=\i as \j using {int(mod(\i,\n)+1)}] in {1,...,\n}{
    \draw[line width=1pt] (H\i) -- (H\j);
}
\foreach \i in {1,...,\n}{
    \draw[line width=1pt]
        (H\i) -- ({60*(\i-1)}:\rout);
}
\node[circle,fill=black,inner sep=1.5pt]
    (TR) at ($(H4)!0.33!(H5)$) {};
\node[circle,fill=black,inner sep=1.5pt]
    (TL) at ($(H4)!0.66!(H5)$) {};
\node[circle,fill=black,inner sep=1.5pt]
    (BL) at ($(H5)!0.33!(H6)$) {};
\node[circle,fill=black,inner sep=1.5pt]
    (BR) at ($(H5)!0.66!(H6)$) {};
\draw[line width=1pt] (BL)to[in=15,out=105](TL);
\draw[line width=1pt] (BR)to[in=15,out=105](TR);
\end{tikzpicture}
\subcaption*{$G_7$}
\end{subfigure}
\begin{subfigure}{0.32\textwidth}
\centering
\begin{tikzpicture}[scale=.75]
\def\rhex{1}
\def\rout{2}
\def\n{6}
\draw[line width=1pt] (0,0) circle (\rout);
\foreach \i in {1,...,\n}{
    \node[circle,fill=black,inner sep=1.5pt]
        (H\i) at ({60*(\i-1)}:\rhex) {};
}
\foreach \i [evaluate=\i as \j using {int(mod(\i,\n)+1)}] in {1,...,\n}{
    \draw[line width=1pt] (H\i) -- (H\j);
}
\foreach \i in {1,...,\n}{
    \draw[line width=1pt]
        (H\i) -- ({60*(\i-1)}:\rout);
}
\node[circle,fill=black,inner sep=1.5pt]
    (TR) at ($(H3)!0.5!(H4)$) {};
\node[circle,fill=black,inner sep=1.5pt]
    (TL) at ($(H4)!0.5!(H5)$) {};
\node[circle,fill=black,inner sep=1.5pt]
    (BL) at ($(H5)!0.5!(H6)$) {};
\node[circle,fill=black,inner sep=1.5pt]
    (BR) at ($(H6)!0.5!(H1)$) {};
\draw[line width=1pt] (BL)to[in=15,out=105](TL);
\draw[line width=1pt] (BR)--(TR);
\end{tikzpicture}
\subcaption*{$G_8$}
\end{subfigure}
\begin{subfigure}{0.32\textwidth}
\centering
\begin{tikzpicture}[scale=.75]
\def\rhex{1}
\def\rout{2}
\def\n{6}
\draw[line width=1pt] (0,0) circle (\rout);
\foreach \i in {1,...,\n}{
    \node[circle,fill=black,inner sep=1.5pt]
        (H\i) at ({60*(\i-1)}:\rhex) {};
}
\foreach \i [evaluate=\i as \j using {int(mod(\i,\n)+1)}] in {1,...,\n}{
    \draw[line width=1pt] (H\i) -- (H\j);
}
\foreach \i in {1,...,\n}{
    \draw[line width=1pt]
        (H\i) -- ({60*(\i-1)}:\rout);
}
\node[circle,fill=black,inner sep=1.5pt]
    (TR) at ($(H4)!0.33!(H5)$) {};
\node[circle,fill=black,inner sep=1.5pt]
    (TL) at ($(H4)!0.66!(H5)$) {};
\node[circle,fill=black,inner sep=1.5pt]
    (BL) at ($(H5)!0.33!(H6)$) {};
\node[circle,fill=black,inner sep=1.5pt]
    (BR) at ($(H5)!0.66!(H6)$) {};
\draw[line width=1pt] (BL)to[in=15,out=105](TR);
\draw[line width=1pt] (BR)to[in=15,out=105](TL);

\end{tikzpicture}
\subcaption*{$G_9$}
\end{subfigure}
\caption{Graphs $G_1,\dots,G_9$ drawn on $\mathbb{P}_2$}\label{fig:Graphs}
\end{figure}
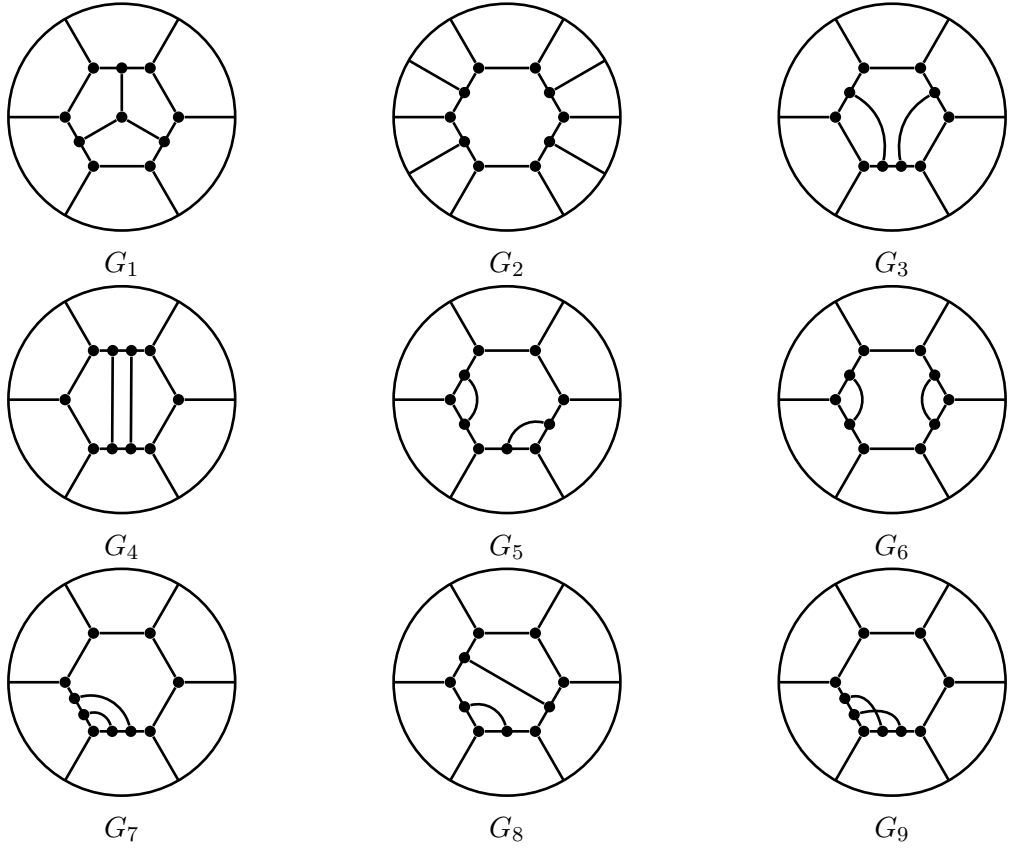

The proof of Theorem \ref{thm:MSR} was first done in \cite{DanilovGrishukhin99} using the terminology of maximal unimodular systems and was done again in our terminology by the second author in \cite{Israel-PhD}. We sketch the latter proof in this section for three reasons. First, it will establish notation we use later in our computations. Second, it will provide explicit embeddings in all but one of the cographic cases that the underlying graph embeds in the real projective plane, and these graph drawings will be used in our computations. Finally, the proof we include here clarifies the distinction between cographic matroids that are maximal among simple cographic matroids and those that are maximal among all simple regular matroids. 

The proof of Theorem \ref{thm:MSR} rests on Seymour's theorem. The first case covers graphic matroids. It is immediate to see that if $H \subseteq G$ is an inclusion of simple graphs with the same vertex set, then there is an inclusion $M(H) \subseteq M(G)$ of the corresponding cycle matroids. In particular, we have the following result usually attributed to Heller \cite{Heller57}(see also \cite[2.1]{DanilovGrishukhin99} and \cite[Theorem 32]{Israel-PhD}).

\begin{lemma}[Graphic case]\label{lem:graphic}
A matroid $M(G)$ is maximal among simple graphic matroids of rank $d$ if and only if $G = K_{d+1}$.
\end{lemma}

The second building block in Seymour's theorem are the cographic matroids. We now list the underlying graphs for the maximal simple cographic matroids (see \cite[Lemmas 2 and 3]{DanilovGrishukhin99} and \cite[Theorem 34 and Lemma 35]{Israel-PhD}): 

\begin{lemma}[Cographic case]\label{lem:cographic}
A matroid $M^*(G)$ is maximal among simple cographic matroids of rank $d$ and is not graphic if and only if $G$ is a $3$-edge-connected, cubic, and non-planar graph with $3(d-1)$ edges.

In particular, the non-graphic and maximal simple cographic matroids of ranks four, five, and six are $M^*(K_{3,3})$, $M^*(G_{5,g})$ for $g \in \{3,4\}$, and $M^*(G_i)$ for $i \in \{1, \ldots, 9\}$, respectively, where the graphs $G_{5,g}$ and $G_i$ are shown in Figures~\ref{fig:2Graphs}~and~\ref{fig:Graphs}, respectively.
\end{lemma}

As we show below, the maximal simple graphic matroid $M(K_7)$ and the maximal simple cographic matroids $M^*(G_i)$ for $i \in \{1,\ldots, 9\} \setminus \{7\}$ are maximal among all simple regular matroids of rank six while $M^*(G_7)$ is not. 

Continuing with Seymour's theorem, we see the sporadic matroid $R_{10}$ precisely when the rank is five. Thus we move on to summarizing the maximal simple regular matroids that arise as $k$-sums for $k \in \{1, 2, 3\}$. Here we have the following structure theorems:

\begin{lemma}[$1$-sums]
No maximal simple regular matroid is a $1$-sum.
\end{lemma}

For the proof, see \cite[Page 412]{DanilovGrishukhin99} or \cite[Theorem 65]{Israel-PhD}.

\begin{lemma}[$2$-sums] The following hold for a maximal simple regular matroid $M$. 
\begin{enumerate}[a)]
    \item If $M$ decomposes as a $2$-sum, then $M$ is a parallel connection of maximal simple regular matroids of strictly lower ranks that sum to $d + 1$.
    \item If moreover $M$ is neither graphic nor cographic, then one of those summands is $R_{10}$ without loss of generality.
    \item If $M$ is not graphic, not cographic, and not $3$-connected, then is the parallel connection $P(M(K_3), R_{10})$ of $M(K_3)$ and $R_{10}$.
\end{enumerate}
\end{lemma}

Note that both $M(K_3)$ and $R_{10}$ are homogeneous in the sense that their automorphism groups act transitively on their ground sets. Therefore there is a unique isomorphism class of matroids obtained by taking a parallel connection of these matroids.

For a proof of the first and third statements, see \cite[Theorems 81 and 82]{Israel-PhD}. For a proof of all three statements, see \cite[Section 7]{DanilovGrishukhin99}, and for a proof of the second statement in particular, see \cite[Lemma 5]{DanilovGrishukhin99}. We note here for convenience that not being $3$-connected is equivalent to being a $1$-sum or a $2$-sum. In the maximal simple regular case, it follows that $M$ is a $2$-sum. In addition, maximality implies that the element of the ground set that is removed when constructing the $2$-sum may be put back into the ground set without breaking the simple or regular properties, so one can see why we work with parallel connections instead of $2$-sums.

\begin{lemma}[$3$-sums]\label{lem:3sum}
If $M$ is a maximal simple regular matroid of rank six that is $3$-connected but not graphic or cographic, then $M$ is the generalized parallel connection $P_T(M(K_5), M^*(K_{3,3}))$ of $M(K_5)$ and $M^*(K_{3,3})$.
\end{lemma}

As in the case of the $2$-sum, the automorphism groups of $M(K_5)$ and $M^*(K_{3,3})$ act transitively on their sets of circuits of length three, so the generalized parallel connection of these matroids along a triangle is well defined up to matroid isomorphism. We also note that the generalized parallel connection may be viewed as a $3$-sum if one first adds three edges in parallel to one of the two matroids along each edge in the $3$-circuit along which the matroids are joined. For a proof of Lemma \ref{lem:3sum}, see \cite[Page 424]{DanilovGrishukhin99} or \cite[Lemma 97 and Section 6.1]{Israel-PhD}. 

Lemmas \ref{lem:graphic} - \ref{lem:3sum} provide a complete list of candidates for the maximal simple regular matroids of rank six, according to Seymour's theorem. To finish the proof of Theorem~\ref{thm:MSR}, one needs to sort out the poset of inclusions among these candidates.

The next lemma covers the most difficult case of establishing inclusion among maximal simple regular matroids of rank six, and we include a proof here. After this is done, we proceed to a proof of Theorem \ref{thm:MSR}. The matroid $R_{12}$ has a representation over $\mathbb{Z}_2$ by removing the last four columns of the matrix $A_{16}$ in (\ref{equ:A16}).

\begin{lemma}[Maximal simple cographic but not maximal simple regular]\label{lem:MSCnotMSR}
The matroid $M^*(G_i)$ has $P_T(M(K_5), M^*(K_{3,3}))$ as a one-element extension if and only if $i = 7$.
\end{lemma}

\begin{proof}
We denote the $16$-element generalized parallel connection by $R_{16}$, which is represented by the matrix $A_{16}$ over $\Z_2$.
\begin{equation}\label{equ:A16}
A_{16}=
\bordersquare{
&& & & & & &  &  &  &  &  &  &  &  &  & \cr
v_1&1&0 &0 &0 &0 &0 & 1 & 1 & 1 & 0 & 0 & 0 & 1 & 1 & 0 & 0 \cr
v_2&0& 1&0 &0 &0 &0 & 1 & 1 & 0 & 1 & 0 & 0 & 1 & 1 & 0 & 0\cr
v_3&0&0 & 1& 0&0 & 0& 1 & 0 & 0 & 0 & 1 & 0 & 0 & 0 & 0 & 1 \cr
v_4&0&0 &0 & 1& 0&0 & 0 & 1 & 0 & 0 & 0 & 1 & 0 & 0 & 0 & 1 \cr
v_5&0&0 &0 &0 & 1&0 & 0 & 0 & 1 & 0 & 1 & 1 & 0 & 1 & 1 & 0 \cr
v_6&0&0 &0 &0 & 0& 1& 0 & 0 & 0 & 1 & 1 & 1 & 0 & 1 & 1 & 0\cr
}
\end{equation}

\noindent Let $E_{16}=E(R_{16})=\{1, 2,\dots, 16\}$ where the columns of $A_{16}$ are indexed in order by $1,2, \dots, 16$.

Fix $f$ in $E_{16}$, and suppose $R_{16}\backslash f$ is isomorphic to $M^*(G_i)$ for some $1 \leq i \leq 9$. Note that $f \not\in \{13,14,15,16\}$ since otherwise $R_{16} \setminus f$ is not cographic (see \cite[Corollary~13.1.6]{Oxley-book}). We may therefore assume that $f \in \{1, 2, \ldots, 12\}$. We show that $i = 7$ by considering the cocircuits of each. 

Note that each of the sets $\{3,7,11,16\}, \{4,8,12,16\}, \{1,2,9,10\},$ $\{5,6,9,10\},$ and $\{1,2,5,6\}$ are cocircuits of $R_{16}$. We label these $C_1^*, C_2^*, \dots, C_5^*$, respectively. It can be verified that each of the sets $C_i^* - f$ is a cocircuit of $R_{16}\backslash f$ with cardinality at most four and that at least one has cardinality three. Moreover, for any $i\neq j$, the sets $C_i^*-f$ and $C_j^*-f$ are pairwise distinct since $C_i^*$ and $C_j^*$ have at most two elements in common. Therefore $R_{16} \setminus f$ has a cocircuit of size three and at least five in all of size at most four. 

On the other hand, the cocircuits of $M^*(G_i)$ are circuits of $M(G_i)$ and hence cycles of $G_i$. By inspection, only $G_7$ has the property of having at least one $3$-cycle and at least five cycles of size at most four. Thus $R_{16}\backslash f\cong M^*(G_i)$ only if $i=7$.

%
%
%

To finish the proof, we show that $M^*(G_7)$ is isomorphic to $R_{16} \backslash 7$. First, note that the latter matroid is represented by deleting the seventh column from $A_{16}$ in (\ref{equ:A16}).
    
Second, we show that this matrix is also a $\Z_2$-representation for $M^*(G_7)$. To see this, we fix a basis of $M^*(G_7)$. The bases of a cographic matroid label the complement of a spanning tree of the graph, one of which we label by $\{e_1,\ldots, e_6\}$ in the graph $G_7$ as shown in Figure \ref{fig:G7}. 

\begin{figure}[h]
    \centering
    \begin{tikzpicture}[scale=1.3,transform shape]
\draw[fill=black] (0,0.5) circle (2pt);
\draw[fill=black] (1,0.5) circle (2pt);
\draw[fill=black] (2,0.5) circle (2pt);
\draw[fill=black] (0,2) circle (2pt);
\draw[fill=black] (1,2) circle (2pt);
\draw[fill=black] (2,2) circle (2pt);
\draw[fill=black] (0.25,3) circle (2pt);
\draw[fill=black] (1.75,3) circle (2pt);
\draw[fill=black] (0, 1.25) circle (2pt);
\draw[fill=black] (2, 1.25) circle (2pt);

\draw[thick] (0,0.5) -- (1,0.5) node[midway, above, yshift=-2pt] {\tiny $f_{_{10}}$};
\draw[thick] (1,0.5) -- (2,0.5) node[midway, above, yshift=-2pt] {\tiny $e_{_3}$};
\draw[thick] (0,0.5) -- (0,2) node[midway, left, yshift=12pt, xshift=3pt] {\tiny $f_{_{13}}$};
\draw[thick] (0,0.5) -- (0,2) node[midway, left, yshift=-12pt, xshift=3pt] {\tiny $f_{_{11}}$};
\draw[thick] (1,0.5) -- (1,2) node[midway, left, yshift=12pt, xshift=4pt] {\tiny $f_{_{14}}$};
\draw[thick] (2,0.5) -- (2,2) node[midway, right, yshift=12pt, xshift=-3pt] {\tiny $f_{_{12}}$};
\draw[thick] (2,0.5) -- (2,2) node[midway, right, yshift=-12pt, xshift=-3pt] {\tiny $e_{_4}$};
\draw[thick] (2,2) -- (0.25,3) node[above, right,xshift=1pt] {\tiny $e_{_1}$};
\draw[thick] (2,2) -- (1.75,3) node[midway, above right, yshift=-5pt, xshift=-3pt] {\tiny$e_{_2}$};
\draw[thick] (1,2) -- (0.25,3) node[midway, below left, yshift=-4pt, xshift=12pt] {\tiny $e_{_5}$};
\draw[thick] (1,2) -- (1.75,3) node[midway, below right, yshift=-4pt, xshift=-10pt] {\tiny $e_{_6}$};
\draw[thick] (0,1.25) -- (2,1.25) node[midway, above, xshift =16pt, yshift=-3pt] {\tiny $f_{_7}$};
\draw[thick] (0,2) -- (0.25,3) node[midway, left, xshift=3pt] {\tiny $f_{_8}$};
\draw[thick] (0,2) -- (1.75,3) node[above, left,xshift=-1pt]{\tiny $f_{_9}$};
\draw [thick] (0,0.5) to [out=-80,in=-100] (2,0.5) node[midway, above, yshift=-5pt, xshift=27pt] {\tiny $f_{_{15}}$};
\end{tikzpicture}
    \caption{Plane drawing of $G_7$ with edge labels}
    \label{fig:G7}
\end{figure}
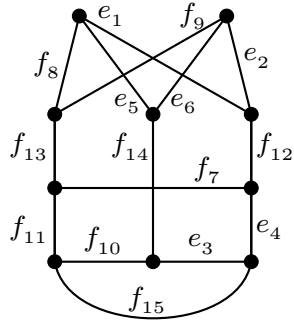

We obtain a $\mathbb{Z}_2$-representation of $M^*(G_7)$ by using the standard basis vectors for the basis $\{e_1,e_2,\dots,e_6\}$. The column corresponding to each $f_j$ is given by putting a $1$ in precisely the entries corresponding to the edges $e_i$ appearing in the unique minimal edge cut contained in $\{e_1,e_2,\dots,e_6,f_j\}$, that is, the fundamental circuit of $M^*(G_7)$ corresponding to this basis and $f_j$. This matrix representation is exactly the matrix given in (\ref{equ:A16}) with the seventh column removed. Hence, $M^*(G_7)\cong R_{16} \backslash 7$.
\end{proof}

We now give a proof of the main theorem in this section.
\begin{proof}[Proof of Theorem \ref{thm:MSR}]
Let $M$ be a maximal simple regular matroid of rank six. If it is graphic, then it is maximal simple graphic and hence is $M(K_7)$ by Lemma \ref{lem:graphic}. Similarly, if $M$ falls into one of the other cases of Seymour's theorem, then $M$ is isomorphic to $M^*(G_i)$ for some $i \in \{1, \ldots, 9\}$, or $M$ is one of $P(M(K_3), R_{10})$ or $P_T(M(K_5), M^*(K_{3,3}))$ by Lemmas \ref{lem:cographic} - \ref{lem:3sum}. It suffices to show that all twelve of these matroids are maximal among all simple regular matroids except for $M^*(G_7)$.

The graphic matroid $M(K_7)$ has more elements in its ground set than the other eleven, so it is clearly maximal. 

Each of the cographic matroids members has $15$ elements, which is more than the $P(M(K_3),R_{10})$, and none of the cographic members are graphic, that is, they are not contained in $M(K_7)$. Moreover, for each $i\neq j$, the matroids $M^*(G_i)$ and $M^*(G_j)$ are pairwise non-isomorphic.
The only possibility for inclusion then is that, for some $i\in \{1,2,\ldots,9\}$, the matroid $M^*(G_i)$ is isomorphic to a single-element deletion of $R_{16}$. By Lemma~\ref{lem:MSCnotMSR}, this is possible if and only if $i = 7$. 

Finally, the maximality of the $2$-sum and the $3$-sum is done directly by analyzing their representations as vector matroids and showing that no vector can be appended to the matrix without losing the simple or regular property. For $P(M(K_3),R_{10})$, this computation is simplified using the maximality and double transitivity of $R_{10}$ and is done in \cite[page 328]{Seymour80} (see also \cite[Lemma 83]{Israel-PhD}). For $R_{16}$, this computation follows from that in \cite[Appendix]{Seymour80} (see also \cite[Exercise 3.1.3]{Oxley-book}).
\end{proof}

\vspace{1em}
\section{Higher cogirths and Theorem \ref{thm:monotonicity}}\label{sec:HigherCogirths}
\vspace{1em}

In this section, we define the notion of $3$-cosystole that we use in this paper and prove the following, which reduces the proof of Theorem \ref{thm:3-cogirth} to the case of maximal simple regular matroids:

\begin{corollary}\label{cor:reduction}
    If $M$ is a regular matroid and $\hat M$ is a maximal simple regular matroid that contains the simplification of $M$, then $\sys_3^*(M) \leq \sys_3^*(\hat M)$.
\end{corollary}

The definitions we give in this section have dual notions in terms of the weighted girth. However, our main results are estimates that hold in the context of fixed ranks (not coranks), so we focus here on cosystoles to avoid confusion.

For a weighted matroid $(M, \mu)$, define

\begin{enumerate}
    \item the weight of a subset $X$ of the ground set $E$ of $M$ by $\mu(X) = \sum_{e \in X} \mu(e)$,  and
    \item the cosystole of $(M, \mu)$ by $\sys^*(M,\mu) = \min \mu(C^*)/\mu(E)$ of $(M, \mu)$ where the minimum runs over cocircuits $C^*$ of $M$.
\end{enumerate}
For a matroid, we define the cosystole 
    \begin{equation}\label{equ:cosys defn}
        \sys^*(M) = \max \sys^*(M, \mu)
    \end{equation}
where the maximum runs over all non-negative, non-zero weight functions $\mu$. Note that the maximum in (\ref{equ:cosys defn}) is well defined since $\sys^*(M, \mu)$ depends continuously on $\mu$ and since scale invariance allows us to restrict the maximum to run over the compact subset probability distributions $\mu$.

For $k \geq 1$, we wish to generalize the cosystole to higher cosystoles that measure weights of the $k$ smallest cocircuits. One natural candidate for this is based on the quantity
    \[\sys_k^*(M, \mu) = \min \sum_{i=1}^{k} \mu(C_i^*)\]
where the minimum runs over all $k$-tuples of pairwise distinct cocircuits $C_1^*,\ldots,C_k^*$. For our purposes, we need a related invariant given by
the same formula except that the minimum runs over $k$-tuples of pairwise distinct cocircuits $C_1^*, \ldots, C_k^*$ having the additional {\it non-inclusion property} that $C_i^*$ is not contained in the union of the other $k-1$ cocircuits for all $1 \leq i \leq k$. 

For $k\in \{1,2\}$, it is trivially true that these notions agree for any probability measure $\mu$. Indeed, the case where $k = 2$ follows since cocircuits, like circuits, cannot properly contain other cocircuits, so the non-inclusion condition is equivalent to the pairwise distinct condition. In this paper, we assume $k = 3$, and we work with triples of cocircuits satisfying the non-inclusion property. 
As above, we then define the $3$-cosystole
    \[\sys_3^*(M) = \max \sys_3^*(M, \mu)\]
by taking a maximum over probability distributions $\mu$ on the ground set of $M$. 


The main results of the rest of this section show that proving upper bounds on $\sys_3^*(M)$ for regular matroids reduces to the simple regular case and then furthermore to the maximal simple regular case.

The first step is to show that $\sys_3^*(M)$ is non-decreasing under equal rank extensions. This proves one of the inequalities in Theorem \ref{thm:monotonicity}.

\begin{lemma}\label{lem:monoton del}
If $M$ is an equal rank, single-element extension of a matroid $N$, then 
    \[\sys_3^*(N)\leq \sys_3^*(M).\]
\end{lemma}

\begin{proof}
    Let $p$ be the element in $E(M)$ such that $M\backslash p= N$. Fix a probability distribution $\nu$ on $N$. Extend this to a probability distribution $\mu$ on $M$ by setting $\mu(p) = 0$, and $\mu = \nu$ otherwise. We claim that $\sys^*_3(N, \nu) \leq \sys_3^*(M, \mu)$. Given this claim, we maximize over probability distributions on $M$ to conclude $\sys_3^*(N, \nu) \leq \sys_3^*(M)$ and then over probability distributions on $N$ to conclude $\sys_3^*(N) \leq \sys_3^*(M)$, as required.

    Fix cocircuits $C_1^*$, $C_2^*$, and $C_3^*$ on $M$ satisfying $C_h^* \not\subseteq C_i^* \cup C_j^*$ for all $\{h, i, j\} = \{1, 2, 3\}$ that realize the minimum in the definition:
        \[\mu(C_1^*) + \mu(C_2^*) + \mu(C_3^*) = \sys_3^*(M, \mu).\]
    It suffices to find cocircuits $D_i^*$ of $N$ with the non-inclusion property and the property that the sum of their weights is at most that of the $C_i^*$. We prove this in each of the following three cases, which up to relabeling of the $C_i^*$ are exhaustive and will complete the proof of the theorem. To start, we fix an element
        \[f_h \in C_h^* \setminus (C_i^* \cup C_j^*)\]
    for all $h \in \{1,2,3\}$ and $\{h, i, j\} = \{1, 2, 3\}$.

    {\bf Case 1}: For all $h \in \{1, 2, 3\}$, there exists a choice of $f_h \neq p$.

    Consider the set $C_h^* - \{p\}$ as a subset of the ground set of $N$. This set is non-empty since it contains $f_h$, and it is a union of cocircuits of $N$ by \cite[Exercise 3.1.2]{Oxley-book}. Therefore, there exists a cocircuit $D_h^* \subseteq C_h^* - \{p\}$ that contains $f_h$ for all $h \in \{1, 2, 3\}$. The triple $\{D_1^*,D_2^*,D_3^*\}$ satisfies the non-inclusion condition and clearly satisfies $\nu(D_h^*) \leq \mu(C_h^*)$ for all $h$, so we have $\sys_3^*(N, \nu) \leq \sys_3^*(M)$ in this case.

    \vspace{1em}
    {\bf Case 2}: $C_3^* \setminus (C_1^* \cup C_2^*) = \{p\}$ and $C_1^* \cap C_2^* \cap C_3^* \neq \emptyset$.

    Fix $g \in C_1^* \cap C_2^* \cap C_3^*$. This time, we choose a cocircuit $D_3^* \subseteq C_3^* - p$ that contains $g$, and we choose cocircuits $C_{i,3}^* \subseteq C_i^* \cup C_3^* - \{g\}$ that contain $f_i$ using the strong cocircuit axiom (see \cite[Exercise 1.1.14]{Oxley-book}). Then we may choose cocircuits $D_i^* \subseteq C_{i,3}^*-p$ that contain $f_i$ for $i \in \{1, 2\}$. Again, it is straightforward to check that the triple $\{D_1^*, D_2^*, D_3^*\}$ has the non-inclusion property, that $\nu(D_1^*) + \nu(D_2^*) \leq \mu(C_1^*) + \mu(C_2^*)$, and that $\nu(D_3^*) \leq \mu(C_3^*)$. Thus the claim holds in this case.

    \vspace{1em}
    {\bf Case 3}: $C_3^* \setminus (C_1^* \cup C_2^*) = \{p\}$ and $C_1^* \cap C_2^* \cap C_3^* = \emptyset$.

    Note that $C_1^* \cap C_3^*$ and $C_2^* \cap C_3^*$ cannot both be empty, since otherwise $C_3^* = \{p\}$ is a coloop, which implies $N$ and $M$ have different ranks. We may swap the roles of $C_1^*$ and $C_2^*$, if necessary, so that $C_1^* \cap C_3^*$ is non-empty and $\mu(C_1^* \cap C_3^*) \geq \mu(C_2^* \cap C_3^*)$. 
    Fix $h_1 \in C_1^* \cap C_3^*$. There exists a cocircuit $C_{13}^* \subseteq C_1^* \cup C_3^* - \{h_1\}$ of $M$ that contains $f_1$ and then a cocircuit $D_1^* \subseteq C_{13}^* - \{p\}$ of $N$ that contains $f_1$. We also fix a cocircuit $D_2^* \subseteq C_2^* - \{p\}$ of $N$ that contains $f_2$ and a cocircuit $D_3^* \subseteq C_3^* - \{p\}$ of $N$ that contains $h_1$. The triple $\{D_1^*,D_2^*,D_3^*\}$ satisfies the non-inclusion property. Moreover, 
        \begin{flalign*}
            \nu(D_1^*)+\nu(D_2^*)+\nu(D_3^*)
            &\leq \mu(C_1^*\backslash C_3^*)+\mu(C_2^*\backslash C_3^*)
            +\mu(C_1^*\cap C_3^*)+3\mu(C_2^*\cap C_3^*)\\
            &\leq \mu(C_1^*\backslash C_3^*)+\mu(C_2^*\backslash C_3^*)
            +2\mu(C_1^*\cap C_3^*)+2\mu(C_2^*\cap C_3^*)\\
            &= \mu(C_1^*)+\mu(C_2^*)+\mu(C_3^*).
        \end{flalign*}
    Hence we again have the proof of the claim in this case.
\end{proof}

Now we prove the other inequality in Theorem \ref{thm:monotonicity}.

\begin{lemma}
If $N$ is a single-element coextension of a matroid $M$ and $r(N)\geq 3$, then 
    \[\sys_3^*(M) \leq \sys_3^*(N).\]
\end{lemma}

\begin{proof}
    Let $N= M/p$ for some $p\in E(M)$. Fix a probability distribution $\mu$ on $M$.
    Define a probability distribution $\mu_-$ on $N$ by $\mu_-(e)=\mu(e)+\tfrac{\mu(p)}{|E|-1}$ for all $e \in E - p$. Choose cocircuits $C_i^*$ of $N$ realizing the minimum 
    \[\sum \mu_-(C_i^*) = \sys_3^*(M/e, \mu_-).\]
    By \cite[3.1.17]{Oxley-book}, the sets $C_i^*$, regarded as subsets of $E$, are also cocircuits in $M$. Moreover, for each element $e$ in $E(M)-\{p\}$, we have $\mu(e)\leq \mu_-(e)$. Therefore,
    \[
        \sys_3^*(M,\mu) \leq \sum \mu(C_i^*) \leq \sum \mu_-(C_i^*).
    \]
    Combining the two displayed equalities, we have $\sys_3^*(M, \mu) \leq \sys_3^*(N,\mu_-)$. As in the previous proof, it follows that $\sys_3^*(M) \leq \sys_3^*(N)$.
\end{proof}

We now finish the proof of Theorem \ref{thm:monotonicity} by showing the third inequality claimed in the theorem.

\begin{lemma}
If $M$ is a regular matroid and $\operatorname{si}(M)$ is its simplification, then
    \[\sys_3^*(M) = \sys_3^*(\operatorname{si}(M)).\]
\end{lemma}

\begin{proof}
Passing from $M$ to its simplification can be achieved by deleting loops one at a time and then by deleting elements one at a time that are parts of two-circuits. Thus, by Lemma~\ref{lem:monoton del}, $\sys_3^*(\operatorname{si}(M))\leq \sys_3^*(M)$. We show that in each of these cases that the $3$-cosystole is also increasing. By induction, this completes the proof.

First suppose $p$ is a loop of $M$. Then $M\setminus p = M/p$ and we can apply the monotonicity under coextensions to conclude
    \[\sys_3^*(M) \leq \sys_3^*(M/p) = \sys_3^*(M\setminus p).\]
This shows that the $3$-cosystole is increasing under deletions of loops.

To finish the proof, we need to show that the $3$-cosystole is increasing under the deletion of an element $e$ that is part of a $2$-circuit $\{e,f\}$. To do this, fix a probability distribution $\mu$ on $M$. Define a probability distribution $\mu_+$ on $E - e$ by $\mu_+(f) = \mu(e) + \mu(f)$ and by $\mu_+ = \mu$ otherwise. Cocircuits in $M$ contain $e$ if and only if they contain $f$. In addition, if $C_1^*$, $C_2^*$, and $C_3^*$ are cocircuits of $M\setminus e$ that realize the minimum in the definition of $\sys_3^*(M\setminus e, \mu_+)$, then we get circuits $\tilde C_i^*$ in $M$ by the formula $\tilde C_i^* = C_i^* \cup e$ if $f \in C_i^*$ and $\tilde C_i^* = C_i^*$ if $f \not\in C_i^*$ that satisfy
    \[\sum \mu(\tilde C_i^*) = \sum \mu_+(C_i^*) = \sys_3^*(M\setminus e,\mu_+).\]
By definition, this implies $\sys_3^*(M,\mu) \leq \sys_3^*(M\setminus e)$ and hence $\sys_3^*(M) \leq \sys_3^*(M\setminus e)$, as claimed.
\end{proof}

The lemmas in this section prove Theorem \ref{thm:monotonicity}, as well as the corollary stated at the beginning of the section.

\vspace{1em}
\section{Computations}
\label{sec:estimates}
\vspace{1em}

In this section, we prove the estimate
    \[\sys_3^*(M) \leq 1\]
for the $11$ simple regular matroids $M$ of rank six given in Theorem \ref{thm:MSR}. Combined with Corollary \ref{cor:reduction}, this will complete the proof of Theorem \ref{thm:3-cogirth}. Along the way, we also prove the estimates on $\sys^*(M)$ claimed in Theorem \ref{thm:estimates<6}. 

The first step is to handle the graphic case.

\begin{lemma}\label{lem:3sys graph}
For $d \geq 3$, we have $\sys_3^*(M(K_{d+1})) = \frac{6}{d+1}$.
\end{lemma}

\begin{proof}
Consider the $d + 1$ cocircuits given by vertex bonds. Note that each edge in $K_{d+1}$ is contained in exactly two of these cocircuits, and note that any three satisfy the non-inclusion property since $d \geq 3$. Therefore, if $\mu$ is any weight function, and if $C_1^*, \ldots, C_{d+1}^*$ denotes these cocircuits ordered by their $\mu$-weight, then we have
    \[\sys_3^*(M(K_{d+1}), \mu) \leq \sum_{i=1}^3 \mu(C_i^*) \leq \frac{3}{d+1} \sum_{i=1}^{d+1} \mu(C_i^*) = \frac{6}{d+1}.\]
Since this estimate holds for any $\mu$, the upper bound follows. 

To see that equality holds, consider the constant weight function $\mu_1$ given by $\mu_1(e) = \binom{d+1}{2}^{-1}$ for all edges $e$. By the edge connectivity of $K_{d+1}$, every cocircuit of $M(K_{d+1})$ contains at least $d$ elements. 
Hence any three cocircuits $C_i^*$ satisfy
    \[\mu(C_1^*) + \mu(C_2^*) + \mu(C_3^*) \geq 3d\binom{d+1}{2}^{-1} = \frac{6}{d+1}.\]
Therefore
    \[\sys_3^*(M(K_{d+1})) 
    \geq \sys_3^*(M(K_{d+1}), \mu_1) \geq  \frac{6}{d+1}.\] 
Combined with the upper bound, this completes the proof.
\end{proof}

The next lemma establishes the 3-cosystole in each of the cographic cases.

\begin{lemma}[$3$-systole of cographic matroids]\label{lem:3sys-cograph}
Let $M$ be the cographic matroid on a graph $G$.
    \begin{enumerate}
        \item If $G = K_{3,3}$, then $\sys_3^*(M) = \tfrac 4 3$.
        \item If $G = G_{5,3}$, then $\sys_3^*(M) = \tfrac{12}{11}$.
        \item If $G = G_{5,4}$, then $\sys_3^*(M) = \tfrac 9 8$.
        \item If $G = G_i$ for some $i \in \{1,\ldots,9\}$, then $\sys_3^*(M) \leq 1$.
    \end{enumerate}
\end{lemma}

\begin{proof}[Proof of Lemma \ref{lem:3sys-cograph} for $K_{3,3}$]
Cocircuits in $M = M^*(K_{3,3})$ are circuits in $M(K_{3,3})$. Let $C_1^*, \ldots, C_9^*$ denote the cocircuits in $M$ with four elements that correspond to the nine $4$-cycles in $K_{3,3}$. Note that every edge in $K_{3,3}$ lies in exactly four of the nine cycles. Moreover, we may relabel if necessary to ensure that each of $\{C_1^*, C_2^*, C_3^*\}$, $\{C_4^*,C_5^*,C_6^*\}$, and $\{C_7^*,C_8^*,C_9^*\}$ satisfy the non-inclusion property.

For any probability distribution $\mu$ on $M$, we have
    \[\mu(C_1^*) + \ldots + \mu(C_9^*) = 4.\]
After relabeling once more, we may assume that $\{C_1^*, C_2^*, C_3^*\}$ has the smallest sum of $\mu$-weights among the three groups we have formed. This implies
    \[\mu(C_1^*) + \mu(C_2^*) + \mu(C_3^*) \leq \frac 4 3,\]
which proves the upper bound claimed in the lemma. To see that equality holds, we note that $\sys_3^*(M)$ is bounded from below by $\sys_3^*(M, \mu_1)$ where $\mu_1$ is the constant weight function. Since all cycles in $G$ have length at least four, the lower bound follows.
\end{proof}

\begin{proof}[Proof of Lemma \ref{lem:3sys-cograph} for $G_{5,3}$]
As in the previous proof, we work with a cycles on the graph $G_{5,3}$. Fix a probability distribution $\mu$ on the edges of the graph. 

There is a unique $3$-cycle $C_0$, exactly three $4$-cycles $C_i$, and exactly six $5$-cycles $D_i$. Their union almost evenly covers the graph four times in the sense that each edge is counted exactly four times. The precise statement is that
    \[2\mu(C_0) + \sum_{i=1}^3 \mu(C_i) + \sum_{i=1}^6 \mu(D_i) = 4.\]
As we did with the previous proof, we make groups of cycles in the graph with the property that no cycle in a group is contained in the union of the others. This time, however, we need one group of three and two groups of four. To form the groups, we relabel so that $D_1 \cap D_2$, $D_3 \cap D_4$, and $D_5 \cap D_6$ have three edges each. In each case, the union $D_i \cup D_{i+1}$ for $i \in \{1,3,5\}$ contains $C_j$ for exactly one choice of $j$ in $\{1,2,3\}$. We group these sets in the form $\{C_j,D_i,D_{i+1}\}$ so that $C_j$ is not contained in $D_i \cup D_{i+1}$. Finally, to two of the groups, we add $C_0$.

By an averaging argument, the group of three has sum of $\mu$-weights at most $\tfrac{12}{11}$ or one of the groups of four has a subset of three with a sum of $\mu$-weights at most this value. In any case, we obtain the estimate $\sys_3^*(M^*(G_{5,3})) \leq \tfrac{12}{11}$. 

To prove that equality holds in this estimate, consider the weight function $\mu_{4,3,1}$ given by giving weight $4/33$ to edges in the three-cycle $C_0$, weight $3/33$ to edges in the four-cycles $C_i$, and weight $1/33$ to the remaining three edges. This weight function has the property that all cycles in the graph have $\mu_{4,3,1}$-weight at least $12/33$. Therefore summing any three gives at least $12/11$. Minimizing over triples satisfying the non-inclusion property then shows $\sys_3^*(M^*(G_{5,3})) \geq \tfrac{12}{11}$. 
\end{proof}

\begin{proof}[Proof of Lemma \ref{lem:3sys-cograph} for $G_{5,4}$]
This graph is isomorphic to the Moebius ladder on four rungs. Using the vertex labels in Figure~\ref{fig:2Graphs}, there are four $4$-cycles $C_i$, labeled by $\{i,i+1,i+4,i+5\}$, and eight $5$-cycles $D_i$, labeled $\{i,i+1,i+2,i+3,i+4\}$, for each $1 \leq i \leq 8$ and taking sums being modulo 8. For each $1 \leq i \leq 8$, take one copy of $C_i$, and one copy of each $D_i$. This gives a total of $16$ cycles whose union counts every edge in the graph exactly six times. We note that $C_i=C_{i+4}$ since the sum is modulo 8. For $j\in\{1,3,5,7\}$ group these cycles by $\{D_j,D_{j+1},C_{j+2},C_{j+3}\}$ with addition modulo 8. Each of these sets has the property that no set is contained in the union of the other three. Hence, any three in each group has the non-inclusion property. 

Therefore for any probability distribution $\mu$ on $G$, the sum of the $\mu$-weights of the $16$ cycles is $6$, and an averaging argument shows that there exists a subset of three of these cycles satisfying the non-inclusion property such that the sum of their $\mu$-weights is at most $3\cdot\tfrac{6}{16} = \tfrac{9}{8}$. Hence $\sys^*(M^*(G_{5,4})) \leq \tfrac 9 8$, as claimed.

As in the previous cases, the lower bound is shown using a specific weight function. In this case, we choose the weight function $\mu_{1,2}$ that is equal to $1/16$ on the eight edges of the 8-cycle in Figure~\ref{fig:2Graphs}, and $2/16$ on each of the edges crossing the boundary. Thus for any cycle $C$ of $G_{5,4}$, the weight $\mu_{1,2}(C)\geq \tfrac 6 {16}$, and the result follows. 
\end{proof}

\begin{proof}[Proof of Lemma \ref{lem:3sys-cograph} for $G_1,\ldots,G_9$]
There is a uniform proof for all $i \neq 9$ based on the fact that these graphs embed in the projective plane, as shown in Figure \ref{fig:Graphs}. Each embedding has six facial cycles, which correspond to six cocircuits in the cographic matroid $M^*(G_i)$. Each edge of the graph is contained in exactly two of these cycles, and any subset of three of these cycles satisfies the non-inclusion property. Therefore the sum of the weights, with respect to any probability distribution, of all six cocircuits is two, and there is a subset of three cocircuits satisfying the non-inclusion property for which the sum of their weights is at most one. Hence $\sys_3^*(M^*(G_i)) \leq 1$ for all $i \neq 9$.

For $i = 9$, we do not have an embedding of $G_9$ into the real projective plane, so we argue as in the previous cases. Let $C_1, \ldots, C_6$ denote the six $4$-cycles in the graph. Each edge of $G_9$ is contained in at most two of these cycles, so we have
    \[\mu(C_1) + \ldots + \mu(C_6) \leq 2\]
for any probability distribution $\mu$. Moreover we may group these cycles into two groups of three that satisfy the non-inclusion property. Therefore one of these groups has the property that the sum of their $\mu$-weights is at most one, as claimed.
\end{proof}

We have completed the proof Theorems \ref{thm:estimates} and \ref{thm:estimates<6} in the graphic and cographic cases, so we move on to the cases of $R_{10}$, $P(M(K_3), R_{10})$, and $P_T(M(K_5), M^*(K_{3,3}))$. For each of these matroids, we work with a binary representation where the elements correspond to the columns of a matrix. 

Before proceeding to the proofs in these cases, we recall some general results about cocircuits in binary matroids. Cocircuits are complements of hyperplanes in general, and hyperplanes in terms of binary representations of matroids can be realized by intersections with vector space hyperplanes, which in turn are described as the orthogonal complements of non-zero vector $v \in \Z_2^d$ where $d$ is the matroid rank. Putting these observations together, one can see that {\it cocircuits of binary matroids are precisely the minimal supports of vectors in the row space} of a binary representation by a matrix whose columns correspond to the elements of the ground set.

\begin{proof}[Proof of Theorem \ref{thm:estimates<6} for $R_{10}$]
Recall the representation of $R_{10}$ given as a submatrix of (\ref{equ:A12}). Consider the cocircuits $C_{12}^*, \ldots, C_{45}^*, C_{51}^*$ given by the supports of the vectors $v_1+v_2,\ldots,v_4+v_5, v_5 +v_1$. Each of these cocircuits has four elements, and every element in the ground set of $R_{10}$ appears in exactly two of these cocircuits. Therefore, for any probability distribution $\mu$ on $R_{10}$, the sum of $\mu$-weights of these five cocircuits equals two. Moreover, these cocircuits also have the property that any three satisfy the non-inclusion condition, so by an averaging argument we obtain the desired upper bound, $\sys_3^*(R_{10}) \leq \tfrac 6 5$. Moreover equality follows easily by considering the constant weight function $\mu_1$, for which $\sys_3^*(R_{10}, \mu_1) = \tfrac 6 5$.
\end{proof}

\begin{proof}[Proof of Theorem \ref{thm:estimates} for the parallel connection] 
Let $M$ denote $P(M(K_3), R_{10})$, which has a binary representation given in (\ref{equ:A12}).
    
     Let $C_0^*$ be the $2$-circuit given by the support of $v_0$, and let $C_1^*, \ldots, C_9^*$ be the cocircuits of length four that are given by supports of vectors in the subspace of the row space given by the span of $v_2,\ldots,v_5$. Such cocircuits exist in $M^*(K_{3,3})$, which is a minor of $R_{10}$ and hence $M$ (see, for example, \cite[Lemma 6.6.9]{Oxley-book}). Alternatively, we note that they correspond to the four vectors $v_i$ for $2 \leq i \leq 5$, the vectors $v_i+v_{i+1}$ for $2 \leq i \leq 4$, and the vectors $v_2+v_3+v_5$ and $v_2+v_4+v_5$.

Given any probability distribution $\mu$ on $M$, we have
    \[4\mu(C_0^*) + \sum_{i=1}^9 \mu(C_i^*) = 4 - 4\mu(f_2) \leq 4.\]
As in previous cases, we can form four groups of three or four, putting one $C_0^*$ in each group, in such a way that the non-inclusion holds for any three in any group. An averaging argument implies that $\sys_3^*(M, \mu) \leq \tfrac{12}{13}$.

To see that equality holds, we consider the weight function given by $\mu_{2,0,1}(f_i) = \tfrac 2 {13}$ for $i \in \{0,1\}$, $\mu_{2,0,1}(f_2) = 0$, and $\mu_{2,0,1}(f_i) = \tfrac 1 {13}$ for $3 \leq i \leq 11$. All cocircuits have $\mu$-weight at least $\tfrac{4}{13}$, so we see that $\sys_3^*(M) \geq \sys_3^*(M, \mu_{2,0,1}) \geq \tfrac{12}{13}$.
\end{proof}

\begin{proof}[Proof of Theorem \ref{thm:estimates} for the generalized parallel connection] 
Let $M$ denote the matroid $P_T(M(K_5), M^*(K_{3,3}))$, which has a binary representation as the vector matroid over $\Z_2$ of the matrix in \eqref{equ:A16}. As in the previous example, we have a $M^*(K_{3,3})$ minor that motivates our choice of cocircuits. 

The first two cocircuits we choose are $C_i^*$ for $i \in \{1,2\}$ corresponding to supports of $v_1$ and $v_2$, and the other nine we consider are motivated by the fact that we again have a $M^*(K_{3,3})$ minor appearing in the last four rows of the matrix. These cocircuits are $C_i^*$ given by the supports of $v_i$ for $i \in \{3,4,5,6\}$ and the other cocircuits $C_i^*$ for $i \in \{7,\ldots,11\}$ are given by the supports of $v_3+v_5$, $v_4+v_6$, $v_5+v_6$, $v_1+v_2+v_3+v_4$, and $v_1+\ldots+v_6$, respectively. 

For any weight function $\mu$, it can be checked that
    \[2\sum_{i=1}^2 \mu(C_i^*) + \sum_{i=3}^{11} \mu(C_i^*) = 4.\]
As in previous proofs, we note that any three cocircuits in each of $\{C_1^*, C_2^*, C_3^*, C_4^*\}$, $\{C_1^*, C_2^*, C_{10}^*\}$, $\{C_5^*, C_6^*, C_7^*\}$, and $\{C_8^*, C_9^*, C_{11}^*\}$ satisfies the non-inclusion property, and an averaging argument shows $\sys_3^*(M) \leq \tfrac{12}{13}$.

As in the previous cases, we can prove equality by considering a specific weight function. In this case, we look at the weight function equal to $0$ on columns $8$, $9$, and $10$ and equal to $\tfrac 1 {13}$ on the others. Every cocircuit has weight at least $\tfrac{4}{13}$. Hence, $\sys^*_3(M)\geq \frac{12}{13}$.
\end{proof}

\bibliographystyle{alpha}
\bibliography{bibfile}

\end{document}